\documentclass[onefignum,onetabnum]{siamart190516}

\usepackage{lipsum}
\usepackage{amsfonts}
\usepackage{graphicx}
\usepackage{epstopdf}
\usepackage{algorithmic}
\usepackage{siunitx}
\usepackage{subcaption}
\usepackage{epstopdf}

\ifpdf
  \DeclareGraphicsExtensions{.eps,.pdf,.png,.jpg}
\else
  \DeclareGraphicsExtensions{.eps}
\fi

\newcommand{\dd}[2]{\frac{\mathrm{d}#1}{\mathrm{d}#2}}
\newcommand{\ddp}[2]{\frac{\partial#1}{\partial#2}}
\newcommand{\D}{\partial D}
\renewcommand{\S}{\mathcal{S}}

\newcommand{\K}{\mathcal{K}}
\newcommand{\de}{\: \mathrm{d}}
\newcommand{\outside}{\mathbb{R}^3\setminus \overline{D}}
\renewcommand{\i}{\mathrm{i}}
\newcommand{\C}{\mathcal{C}}
\renewcommand*{\Re}{\operatorname{Re}}
\renewcommand*{\Im}{\operatorname{Im}}
\newcommand{\Con}{\mathcal{C}}

\graphicspath{ {../figures/} }

\newsiamremark{remark}{Remark}
\newsiamremark{hypothesis}{Hypothesis}
\crefname{hypothesis}{Hypothesis}{Hypotheses}
\newsiamthm{claim}{Claim}

\headers{Asymptotic biomimicry}{H. Ammari and B. Davies}

\title{Asymptotic links between signal processing, acoustic metamaterials and biology}

\author{Habib Ammari\thanks{Department of Mathematics, ETH Zurich, Rämistrasse 101, CH-8092 Zürich, Switzerland
  (\email{habib.ammari@math.ethz.ch}).}
\and Bryn Davies\thanks{Department of Mathematics, Imperial College London, 180 Queen's Gate, London, SW7 2AZ, UK} (\email{bryn.davies@imperial.ac.uk}).}

\usepackage{amsopn}

\ifpdf
\hypersetup{
  pdftitle={Asymptotic biomimicry},
  pdfauthor={H. Ammari and B. Davies}
}
\fi

\begin{document}

\maketitle

\begin{abstract}
	Biomimicry is a powerful science that takes inspiration from nature's innovative solutions to challenging problems. In this work, we use asymptotic methods to develop the mathematical foundations for the exchange of design inspiration and features between biological hearing systems, signal processing algorithms and acoustic metamaterials. Our starting point is a concise asymptotic analysis of high-contrast acoustic metamaterials. We are able to fine tune this graded structure to mimic the biomechanical properties of the cochlea, at the same scale. We then turn our attention to developing a biomimetic signal processing algorithm. We use the response of the cochlea-like metamaterial as an initial filtering layer and then add additional biomimetic processing stages, designed to mimic the human auditory system's ability to recognise the global properties of natural sounds. This demonstrates the three-way exchange of ideas that, thanks to our analysis, is possible between signal processing, metamaterials and biology.
\end{abstract}



\section{Introduction}

This work uses asymptotic methods to explore links between the human auditory system, acoustic metamaterials and signal processing algorithms. Similar links are often exploited by engineers looking for design inspiration in a practice known as \emph{biomimicry}, whereby systems are designed to mimic the function of biological systems. Since the performance of biological systems typically represents the gold standard against which other solutions are compared, this has proved a successful approach to innovation in many different fields of science and engineering \cite{benyus1997biomimicry}. This work will provide the mathematical underpinning for links between the fields, exploiting their shared biomimetic goal. It will facilitate a trilateral exchange of ideas (as depicted in Fig.~\ref{fig:workflow}), which has the potential to accelerate progress.

\begin{figure}
	\centering
	\includegraphics{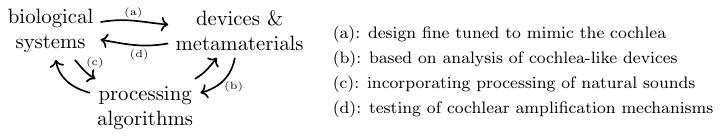}
	\caption{Biomimicry in action. This work sets out the mathematical foundations for the exchange of design principles and features between biological hearing systems, artificial sound-filtering devices and signal processing algorithms.}
	\label{fig:workflow}
\end{figure}

In the world of signal processing, biomimetic solutions have been longstanding. For example, the multi-scale nature of the wavelet transform, which is linear but has logarithmic frequency scaling, replicates the response of the human auditory system \cite{daubechies1992ten, yang1992auditory}. Additionally, efforts have been made to develop approaches that replicate the human auditory system's invariance to certain transformations. For example, the scattering transform is based on the wavelet transform but is stable to translations and $C^1$ time-warping deformations \cite{anden2014deep, mallat2012group}. As well as wavelets, other kernels have been explored for use in convolutional signal transforms, such as windowed Fourier modes \cite{alm2002time, cohen1995time} and learned kernel functions \cite{smith2006efficient}. Although less prevalent, gammatone kernels have been observed to replicate human auditory function particularly well. This has been derived from several sources, including approximations of response curves from psychoacoustic experiments \cite{patterson1976auditory, patterson1974auditory, patterson1988efficient}, models of cochlear receptor cells \cite{hewitt1994computer} and direct observations of the cochlear basilar membrane, whose response can be fitted to sums of gammatones \cite{bell2018cochlear, flanagan1960models}. Since gammatones also have favourable properties such as causality, non-invariance to time reversal and ease of implementation \cite{lyon2017human}, they are often used to replicate human auditory processing.

A community of researchers designing graded metamaterials that mimic the function of the cochlea has emerged recently \cite{davies2019fully, davies2020hopf, karlos2020cochlea, rupin2019mimicking, zhao2019metamaterial}. Metamaterials are artificial materials with complex microstructures that can be designed to achieve remarkable and powerful properties \cite{craster2012acoustic}. In this work, we are interested in graded metamaterials whose material parameters vary adiabatically as a function of one spatial dimension. This property allows them to separate frequencies via the \emph{rainbow effect}. This was first observed in the setting of optics \cite{tsakmakidis2007trapped} and has since been studied in many settings, including water waves \cite{bennetts2018graded}, plasmonics \cite{jang2011plasmonic} and acoustics \cite{zhu2013acoustic}. This phenomenon is qualitatively similar to how frequencies are filtered by a graded membrane in the cochlea and researchers have explored the extent to which the metamaterials' designs can be fine tuned to match the cochlear response. Several different types of resonators have been used, including vibrating reeds \cite{bell2012resonance}, quarter-wavelength pipe resonators \cite{rupin2019mimicking}, Helmholtz oscillators \cite{karlos2020cochlea, zhao2019metamaterial} and Minnaert cavities \cite{davies2019fully, davies2020hopf, davies2021robustness}.

In this work, we study an acoustic metamaterial composed of high-contrast material cavities, referred to here as \emph{subwavelength resonators} and sometimes known as \emph{Minnaert cavities}. These are material cavities whose density is significantly less than that of the background medium, meaning they experience resonance in response to wavelengths much greater than their size. The inspiration for this setup comes from Minnaert's study of the acoustic properties of air bubbles in water \cite{minnaert1933musical}. Customisable structures with these properties can be created by injecting bubbles into polymer gels \cite{leroy2009design, leroy2009transmission}. Graded systems of subwavelength resonators that mimic the frequency separation of the cochlea were studied in \cite{davies2019fully}. A two-dimensional model was studied and it was shown that the structure's properties can be chosen such that the frequency-position (tonotopic) map of the cochlea can be reproduced, as well as its characteristic travelling wave behaviour. It was also shown that the function of this device is robust with respect to imperfections in its design, to an extent that is similar to the cochlea itself \cite{davies2021robustness}. This transfer of design ideas and inspiration is labelled (a) in Fig.~\ref{fig:workflow}.

We will present a design for a three-dimensional graded metamaterial that mimicks the frequency separation of the cochlea. This builds on the two-dimensional analysis of \cite{davies2019fully} and uses an asymptotic characterisation in terms of the \emph{generalized capacitance matrix}, as introduced in \cite{ammari2021capacitance}. This approximation provides a very efficient computational framework for optimising the properties of the structure. We then derive a leading-order time-domain decomposition for the scattered field, which takes the form of a modal decomposition with coefficients given by convolutions with gammatone kernels. The signal transform given by these coefficients has favourable continuity properties, which we explore in this work, in addition to its biomimetic derivation. The deduction of a signal processing approach from the bio-inspired metamaterial corresponds to (b) in Fig.~\ref{fig:workflow}.

In the final part of this paper, we explore link (c) in Fig.~\ref{fig:workflow}, to exemplify the trilateral exchange of ideas that is possible. Given the signal transform that was deduced from the cochlea-inspired metamaterial, we use other properties of the human auditory system to design additional processing steps to reveal more sophisticated features of sounds. In particular, we will exploit the fact that the human auditory system is adapted to the properties of certain important inputs and that enhanced neural responses have been observed in response to behaviourally significant sounds such as animal calls, human vocalisations and environmental sounds \cite{theunissen2014neural, yu2005preference}. These sounds are often known collectively as \emph{natural sounds} and have been observed to satisfy certain long-range statistical properties, which we interpret as their defining characteristic \cite{voss19751, attias1997temporal, attias1998coding}. We will develop an approach that takes the output from an array of gammatone filters (as deduced from a cochlea-inspired metamaterial) and extracts the parameters that describe the global properties of natural sounds. It is beyond the scope of the present work to explore applications of this representation. The discussion presented here is intended as an example of the innovation that is possible, given a framework like the one developed in this work.


A valuable observation about the cochlea is that it is an active sensor and uses an amplification mechanism in its function. While we will only consider passive metamaterials in this work, this property can be replicated by introducing amplification to our system. This was studied theoretically in \cite{davies2020hopf} and in experiments on scaled-up arrays by \cite{rupin2019mimicking}. This is a promising line of investigation since there are many open questions about the nature of cochlear amplification \cite{hudspeth2010critique}. Cochlea-inspired active metamaterials offer a convenient analogue environment in which to conduct real-time experiments to test theories, without the need for expensive simulations or challenging and invasive experiments on biological specimens. This line of investigation is labelled (d) in Fig.~\ref{fig:workflow}.

In this paper, we will explore links (a)-(c) in Fig.~\ref{fig:workflow} one by one. In Section~\ref{sec:metamaterial} we develop a concise, rigorous approach for studying an array of subwavelength resonators and design a structure that mimics the biomechanical properties of the cochlea. In Section~\ref{sec:transforms} we examine the use of an array of gammatone filters, deduced from the cochlea-inspired metamaterial, as a signal processing algorithm. We study its stability properties and its behaviour when cascaded. Finally, in Section~\ref{sec:natsounds} we focus our attention on the class of natural sounds and propose a biomimetic approach that extracts the global properties of a sound from the output of the cochlea-inspired metamaterial.

\section{Analysis of a cochlea-inspired metamaterial} \label{sec:metamaterial}

\subsection{Problem setting}

We will study a simple metamaterial composed of $N\in\mathbb{N}$ disjoint bounded inclusions, which represent the resonators and will be labelled as $D_1,D_2,\dots,D_N\subset\mathbb{R}^3$, surrounded by an unbounded homogeneous background medium. We will assume that the boundaries are all in $C^{1,\alpha}$ for some $0<\alpha<1$ and will write $D=D_1\cup\dots\cup D_N$ to denote the entire structure. In order to replicate the spatial frequency separation of the cochlea, we are interested in the case where the resonators are arranged in a line and the array has a size gradient, meaning each resonator is slightly larger than the previous, as depicted in Fig.~\ref{fig:geom}.

\begin{figure}
	\centering
	\includegraphics[width=0.9\linewidth]{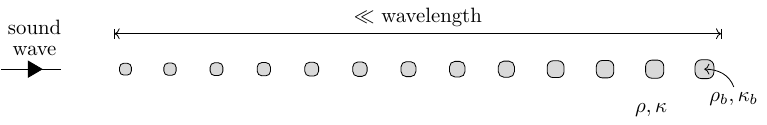}
	\caption{A graded array of subwavelength resonators can be designed to mimic the biomechanical properties of the cochlea in response to a sound wave. The large density contrast $\delta:=\rho_b/\rho\ll1$ is crucial for the subwavelength resonant response of the structure.}
	\label{fig:geom}
\end{figure}

We will denote the density and bulk modulus of the material within the bounded regions $D$ by $\rho_b$ and $\kappa_b$, respectively. The corresponding parameters for the background medium are $\rho$ and $\kappa$. The wave speeds in $D$ and $\outside$ are given by
\begin{equation*}
v_b=\sqrt{\frac{\kappa_b}{\rho_b}} \quad\text{and} \quad v=\sqrt{\frac{\kappa}{\rho}}.
\end{equation*}
We also define the dimensionless contrast parameter
\begin{equation*}
\delta=\frac{\rho_b}{\rho}.
\end{equation*}
We will assume that $\delta\ll1$, meanwhile $v_b=O(1)$ and $v=O(1)$. This high-contrast assumption will give the desired subwavelength behaviour and facilitate an asymptotic analysis in terms of $\delta$. See \cite{ammari2020review} for a review of other applications of structures satisfying this high-contrast criterion.

\subsection{Boundary integral operators}

We will use boundary integral formulations to study the scattering of time-harmonic acoustic waves by the resonator array $D$. In particular, we will represent solutions using the Helmholtz single layer potential $\S_{D}^k$. The Helmholtz single layer potential associated to the domain $D$ and wavenumber $k\in\mathbb{C}$ is defined, for some density function $\varphi\in L^2(\D)$, as
\begin{equation}
\S_{D}^k[\varphi](x):=\int_{\D} G(x-y,k)\varphi(k)\de\sigma(y), \qquad x\in\mathbb{R}^3,
\end{equation}
where $G$ is the outgoing Helmholtz Green's function, given by
\begin{equation} \label{eq:green}
G(x,k):=-\frac{1}{4\pi|x|}e^{\i k|x|}.
\end{equation}
We will use the single layer potential to reduce the problem of solving a Helmholtz problem in $\mathbb{R}^3$ to a problem posed in terms of density functions evaluated on the boundary $\D$.

We will be interested in the low frequency regime and seek solutions as $k\to0$. We will make use of the asymptotic expansion
\begin{equation} \label{eq:S_expansion}
\S_{D}^k[\varphi]=\S_{D}[\varphi]+k\S_{D,1}[\varphi]+O(k^2),
\end{equation}
where $\S_{D}:=\S_D^0$ (\emph{i.e.} the Laplace single layer potential) and
\begin{equation*}
\S_{D,1}[\varphi](x):=\frac{1}{4\pi\i}\int_{\D} \varphi(y) \de\sigma(y).
\end{equation*}
One crucial property to note is that $\S_D$ is invertible as a map $L^2(\D)\to H^1(\D)$, where $H^1(\D)$ is the space of square integrable functions on $\D$ whose weak first derivatives exist and are square integrable. Further properties of the single layer potential and related operators can be found in \emph{e.g.} \cite{ammari2018mathematical}. We will need one additional piece of boundary-integral machinery: the Neumann-Poincar\'e operator, which is valuable since it can be used to describe how the derivative of the single layer potential varies across the boundary $\D$. The Neumann-Poincar\'e operator associated to the domain $D$ and wavenumber $k\in\mathbb{C}$ is defined, for a density $\varphi\in L^2(\D)$, as
\begin{equation}
\K_{D}^{k,*}[\varphi](x):=\int_{\D} \frac{\partial G(x-y,k)}{\partial \nu(x)}\varphi(y)\de\sigma(y), \qquad x\in\D,
\end{equation}
where $G$ is the outgoing Helmholtz Green's function \eqref{eq:green} and $\partial/\partial\nu$ denotes the outward normal derivative on $\D$. This is related to the single layer potential by the conditions
\begin{equation} \label{eq:jump}
\ddp{}{\nu}\S_{D}^k[\varphi]\big|_\pm(x)= \left(\pm\frac{1}{2}I+\K_{D}^{k,*}\right)[\varphi](x),\qquad \varphi\in L^2(\D),\,x\in\D,
\end{equation}
where the subscripts $+$ and $-$ denote evaluation from outside and inside the boundary $\D$, respectively, and $I$ is the identity operator on $L^2(\D)$.

As is the case for the single layer potential, the Neumann--Poincar{\'e} operator has a helpful asymptotic expansion at low frequencies. In particular, we have that as $k\to0$
\begin{equation} \label{eq:K_expansion}
\K_{D}^{k,*}[\varphi]=\K_{D}^{*}[\varphi]+k\K_{D,1}[\varphi] + k^2\K_{D,2}[\varphi] +k^3\K_{D,3}[\varphi]+O(k^4),
\end{equation}
where $\K_{D}^{*}:=\K_{D}^{0,*}$, $\K_{D,1}=0$, 
\begin{equation}
\K_{D,2}[\varphi](x):=\frac{1}{8\pi}\int_{\D}\frac{(x-y)\cdot \nu(x)}{|x-y|}\varphi(y)\de\sigma(y)
\end{equation}
and
\begin{equation}
\K_{D,3}[\varphi](x):=\frac{\i}{12\pi}\int_{\D} (x-y)\cdot \nu(x)\varphi(y)\de\sigma(y).
\end{equation}

Several of the operators in the expansion \eqref{eq:K_expansion} can be simplified when integrated over all or part of the boundary $\D$. As proved in \cite[Lemma~2.1]{ammari2017double}, it holds for any $\varphi\in L^2(\D)$ and $i=1,\dots,N$ that
\begin{equation} \label{K_properties}
\begin{split}
\int_{\D_i}\left(-\frac{1}{2}I+\K_D^{*}\right)[\varphi]\de\sigma=0,
\qquad&\int_{\D_i}\left(\frac{1}{2}I+\K_D^{*}\right)[\varphi]\de\sigma=\int_{\D_i}\varphi\de\sigma,\\
\int_{\D_i} \K_{D,2}[\varphi]\de\sigma=-\int_{D_i}\S_D[\varphi]\de x
\quad&\text{and}\quad\int_{\D_i} \K_{D,3}[\varphi]\de\sigma=\frac{\i|D_i|}{4\pi}\int_{\D}\varphi\de\sigma.
\end{split}
\end{equation}

\subsection{Subwavelength resonance} \label{sec:pure_tones}

Suppose that the incoming signal is a plane wave parallel to the $x_1$-axis with angular frequency $\omega$ given by $u^{in}=A\cos(kx_1-\omega t)$, where $k=\omega/v$ is the wavenumber. Then, the scattered pressure field is given by $\Re(u(x,\omega)e^{-\i\omega t})$ where $u$ satisfies the Helmholtz equation
\begin{equation} \label{eq:helmholtz_equation}
\begin{cases}
\left( \Delta + k^2 \right)u = 0 & \text{in} \ \outside, \\
\left( \Delta + k_b^2 \right)u = 0 & \text{in} \ D,
\end{cases}
\end{equation}
where $k=\omega/v$ and $k_b=\omega/v_b$, along with the transmission conditions
\begin{equation} \label{eq:transmission}
\begin{cases}
u_+ - u_- = 0 & \text{on} \ \D,\\
\frac{1}{\rho} \ddp{u}{\nu_x}\big|_+ - \frac{1}{\rho_b} \ddp{u}{\nu_x}\big|_- = 0 & \text{on} \ \D,
\end{cases}
\end{equation} 
and $u-u^{in}$ satisfying the Sommerfeld radiation condition in the far field, which ensures that energy radiates outwards \cite{ammari2018mathematical}. We will characterise solutions to this scattering problem in terms of its subwavelength resonant modes.

\begin{definition}[Subwavelength resonant frequency]
	We define a {subwavelength resonant frequency} to be $\omega=\omega(\delta)$ such that:
	\begin{itemize}
		\item[(i)] there exists a non-trivial solution to \eqref{eq:helmholtz_equation} which satisfies \eqref{eq:transmission} and the radiation condition when $u^{in}=0$,
		\item[(ii)] $\omega$ depends continuously on $\delta$ and is such that $\omega(\delta)\to0$ as $\delta\to0$.
	\end{itemize}
\end{definition}

It is well known that the spectrum of the scattering problem \eqref{eq:helmholtz_equation} is symmetric in the imaginary axis. In particular, we have the following result from \cite{dyatlov2019mathematical}, where $z^*$ is used to denote the complex conjugate of $z\in\mathbb{C}$. 
\begin{lemma}
	If $\omega\in\mathbb{C}$ is a subwavelength resonant frequency, then $-\omega^*$ is also a subwavelength resonant frequency with the same multiplicity.
\end{lemma}
The existence of subwavelength resonant frequencies is given by the following lemma, which was proved in \cite{davies2019fully} using the asymptotic theory of Gohberg and Sigal \cite{gohberg2009holomorphic, ammari2018mathematical}.

\begin{lemma}
	A system of $N$ subwavelength resonators exhibits $N$ subwavelength resonant frequencies with positive real parts, up to multiplicity.
\end{lemma}
We will use the notation $\omega_n^\pm$ to denote the resonant frequencies for $n=1,\dots,N$, where $\Re(\omega_n^+)>0$ and $\omega_n^-=-(\omega_n^+)^*$.

\subsection{The capacitance matrix}

Our approach to finding the system's resonant frequencies is to study the \emph{generalized capacitance matrix}, which offers a rigorous discrete approximation to the differential problem. We will see that the eigenstates of this $N\times N$-matrix characterise, at leading order in $\delta$, the resonant modes of the system. This concise characterisation of the scattering problem will be the key that allows us to fine tune the resonator array such that it replicates the action of the cochlea. This approach has previously been used for various high-contrast settings, including non-Hermitian and time-modulated systems. For a comprehensive review see \cite{ammari2021capacitance}.

In order to introduce the notion of capacitance, we define the functions $\psi_j$, for $j=1,\dots,N$, as 
\begin{equation}
\psi_j:=\S_D^{-1}[\chi_{\D_j}],
\end{equation}
where $\S_D^{-1}$ is the inverse of $\S_D:L^2(\D)\to H^1(\D)$ and $\chi_A:\mathbb{R}^3\to\{0,1\}$ is used to denote the characteristic function of a set $A\subset\mathbb{R}^3$.	The capacitance matrix $C=(C_{ij})$ is defined, for $i,j=1,\dots,N$, as
\begin{equation}
C_{ij}:=-\int_{\D_i} \psi_j\de\sigma.
\end{equation}
It is known that $C$ is a symmetric matrix \cite{ammari2018mathematical}. In order to capture the behaviour of an asymmetric array of resonators we need to introduce the \emph{generalized} capacitance matrix $\C=(\C_{ij})$, given by
\begin{equation} \label{GCMdefn}
\C_{ij}:=\frac{1}{|D_i|} C_{ij},
\end{equation}
which accounts for the differently sized resonators.

\begin{remark} \label{rmk:Cvol_symposdef}
	While the weighting in $\C$ hides some its underlying structure, it is the product of the symmetric matrix $C$ with the diagonal matrix of inverse volumes. This means, in particular, that it has a basis of eigenvectors.
\end{remark}

\begin{remark}
	Capacitance coefficients have a long history of applications in electrostatics, where they are used to relate the potentials and charges on a system of conductors \cite{lekner2011capacitance}. While the interpretation is slightly different here, the intuition that they represent the extent of the interactions between each pair of objects in a many-body system remains the same (in our case, this is at leading order in the contrast parameter $\delta$).
\end{remark}

\begin{remark}
The generalized capacitance matrix $\C$ is often defined with the material parameters included in the weights in \eqref{GCMdefn}, such that it is able to describe systems with different material parameters inside each material inclusion \cite{ammari2021capacitance}. We omit this here since we are considering a structure where only the size of the inclusions is varied. However, this generalized capacitance matrix approximation shows that similar behaviour could be obtained by grading the material parameters instead.  
\end{remark}

Let us define the functions $S_n^\omega$, for $n=1\dots,N$, as
\begin{equation}
S_n^\omega(x) := \begin{cases}
\S_{D}^{k}[\psi_n](x), & x\in\outside,\\
\S_{D}^{k_b}[\psi_n](x), & x\in D.\\
\end{cases}
\end{equation}
We will express the solutions to the Helmholtz scattering problem \eqref{eq:helmholtz_equation} as a decomposition in terms of these functions. The significance of the capacitance matrix $\C$ will become apparent through this approach.

\begin{lemma} \label{lem:modal}
	The solution to the scattering problem \eqref{eq:helmholtz_equation} can be written, for $x\in\mathbb{R}^3$ and incoming wave $u^{in}=Ae^{\i kx_1}$, as
	\begin{equation*} 
	u(x)-Ae^{\i kx_1} = \sum_{n=1}^N q_nS_n^\omega(x) - \S_D^k\left[\S_D^{-1}[Ae^{\i kx_1}]\right](x) + O(\omega),
	\end{equation*} 
	for constants $q_n$ which satisfy, up to an error of order $O(\delta \omega+\omega^3)$, the problem
	\begin{equation*}\label{eq:eval_C}
	\left({\omega^2}I_N-{v_b^2\delta}\,\C\right)\begin{pmatrix}q_1\\ \vdots\\q_N\end{pmatrix}
	=
	{v_b^2\delta}\begin{pmatrix} \frac{1}{|D_1|} \int_{\D_1}\S_D^{-1}[Ae^{\i kx_1}]\de\sigma \\ \vdots\\
	\frac{1}{|D_N|}\int_{\D_N}\S_D^{-1}[Ae^{\i kx_1}]\de\sigma \end{pmatrix},
	\end{equation*}
	where $I_N$ is the $N\times N$ identity matrix.
\end{lemma}
\begin{proof}
	The solutions can be represented as 
	\begin{equation} \label{eq:layer_potential_representation}
	u(x) = \begin{cases}
	Ae^{\i kx_1}+\S_{D}^{k}[\psi](x), & x\in\outside,\\
	\S_{D}^{k_b}[\phi](x), & x\in D,
	\end{cases}
	\end{equation} 
	for some surface potentials $(\phi,\psi)\in L^2(\D)\times L^2(\D)$, which must be chosen so that $u$ satisfies the transmission conditions across $\D$. 
	
	Using \eqref{eq:jump}, we see that in order to satisfy the transmission conditions on $\D$ the densities $\phi$ and $\psi$ must satisfy, for $x\in \D$,
	\begin{align*}
	\S_{D}^{k_b}[\phi](x)-\S_{D}^{k}[\psi](x)&=Ae^{\i kx_1}, \\
	\left(-\frac{1}{2}I+\K_{D}^{k_b,*}\right)[\phi](x)-\delta\left(\frac{1}{2}I+\K_{D}^{k,*}\right)[\psi](x)&=\delta \ddp{}{\nu}(Ae^{\i kx_1}).
	\end{align*}
	Using the asymptotic expansions \eqref{eq:S_expansion} and \eqref{eq:K_expansion}, we can see that 
	\begin{equation*}\label{eq:psi}
	\psi=\phi-\S_D^{-1}[Ae^{\i kx_1}]+O(\omega),
	\end{equation*} 
	and, further, that up to an error of order $O(\delta\omega+\omega^3)$
	\begin{equation} 
	\left(-\frac{1}{2}I+\K_D^*+\frac{\omega^2}{{v}_b^2}\K_{D,2}-\delta\left(\frac{1}{2}I+\K_D^*\right)\right)[\phi]=-\delta \left(\frac{1}{2}I+\K_D^*\right)\S_D^{-1}[Ae^{\i kx_1}]. \label{eq:phi}
	\end{equation}
	At leading order, \eqref{eq:phi} says that $\left(-\frac{1}{2}I+\K_D^{*}\right)[\phi]=0$ so, in light of the fact that $\{\psi_1,\dots,\psi_N\}$ forms a basis for $\ker\left(-\frac{1}{2}I+\K_D^{*}\right)$, the solution can be written as
	\begin{equation} \label{eq:psi_basis}
	\phi=\sum_{n=1}^N q_n\psi_n+O(\omega^2+\delta),
	\end{equation}
	for constants $q_1,\dots,q_N=O(1)$.
	Then, integrating \eqref{eq:phi} over $\D_i$, for $1\leq i\leq N$, and using the properties \eqref{K_properties} gives us that
	\begin{equation*}
	-\omega^2\int_{D_i}\S_D[\phi]\de x -{v}_b^2\delta\int_{\D_i}\phi\de\sigma=-{v_b^2\delta}\int_{\D_i}\S_D^{-1}[Ae^{\i kx_1}]\de\sigma+O(\delta\omega+\omega^3).  \label{eq:D}
	\end{equation*}
	Substituting the ansatz \eqref{eq:psi_basis} gives the desired result. \qed
\end{proof}

The resonant modes of the system are non-trivial solutions in the case that $u^{in}=0$. Thus, from Lemma~\ref{lem:modal}, we can see that resonance occurs when $\omega^2/v_b^2\delta$ is an eigenvalue of $\C$, at leading order. We can continue this argument to higher orders and obtain the following theorem, which characterises the subwavelength resonant frequencies of the system.

\begin{theorem} \label{thm:res}
	As $\delta \rightarrow 0$, the subwavelength resonant frequencies satisfy the asymptotic formula
	$$\omega_n^\pm = \pm\sqrt{v_b^2\lambda_n\delta} -\i\tau_n\delta+ O(\delta^{3/2}),$$
	for  $n = 1,\dots,N$, where $\lambda_n$ are the eigenvalues of the generalized capacitance matrix $\C$ and $ \tau_n$ are non-negative real numbers given by 
	\begin{equation*}
	\tau_n= \frac{v_b^2}{8\pi v} \frac{\underline{v}_n\cdot C J C\underline{v}_n}{\| \underline{v}_n\|_D^2},
	\end{equation*}
	where $C$ is the capacitance matrix, $J$ is the $N\times N$ matrix of ones, $\underline{v}_n$ is the eigenvector associated to $\lambda_n$ and we use the norm $\| x\|_D:=\big(\sum_{i=1}^N |D_i| x_i^2\big)^{1/2}$.
\end{theorem}
\begin{proof}
	If $u^{in} = 0$, we find from Lemma~\ref{lem:modal} that there is a non-zero solution $q_1,\dots,q_N$ to the eigenvalue problem \eqref{eq:eval_C} when $\omega^2/v_b^2\delta$ is an eigenvalue of $\C$, at leading order.
	
	To find the imaginary part, we adopt the ansatz
	\begin{equation} \label{eq:omega_ansatz}
	\omega_n^\pm=\pm\sqrt{v_b^2\lambda_n\delta} -\i\tau_n\delta+ O(\delta^{3/2}),
	\end{equation}
	where $\lambda_n$ is an eigenvalue of $\C$ and $\tau_n$ is a real number. Using the expansions \eqref{eq:S_expansion} and \eqref{eq:K_expansion} with the representation \eqref{eq:layer_potential_representation} we have that
	\begin{equation*}
	\psi=\phi+\frac{k_b-k}{4\pi\i}\left(\sum_{n=1}^N\psi_n\right)\int_{\D}\phi\de\sigma+O(\omega^2),
	\end{equation*} 
	and, hence, that
	\begin{equation}
	\begin{split}
	&\left(-\frac{1}{2}I+\K_D^*+k_b^2\K_{D,2}+k_b^3\K_{D,3}-\delta\left(\frac{1}{2}I+\K_D^*\right)\right)[\phi]\\&\qquad\qquad\qquad\qquad\qquad\qquad-\frac{\delta(k_b-k)}{4\pi\i}\left(\sum_{n=1}^N\psi_n\right)\int_{\D}\phi\de\sigma=O(\delta\omega^2+\omega^4).
	\end{split}
	\end{equation}
	We then substitute the decomposition \eqref{eq:psi_basis} and integrate over $\D_i$, for $i=1,\dots,N$, to find that, up to an error of order $O(\delta\omega^2+\omega^4)$, it holds that
	\begin{equation*}
	\bigg(-\frac{\omega^2}{v_b^2}-\frac{\omega^3\i}{4\pi v_b^3}JC+\delta \C+\frac{\delta\omega\i}{4\pi}\bigg(\frac{1}{v_b}-\frac{1}{v}\bigg)\C JC \bigg)
	\underline{q}=0,
	\end{equation*}
	where $J$ is the $N\times N$ matrix of ones (\emph{i.e.} $J_{ij}=1$ for all $i,j=1,\dots,N$). 
	Then, using the ansatz \eqref{eq:omega_ansatz} for $\omega_n^\pm$ we see that, if $\underline{v}_n$ is an eigenvector corresponding to $\lambda_n$, it holds that
	\begin{equation}
	\tau_n= \frac{v_b^2}{8\pi v} \frac{\underline{v}_n\cdot C J C\underline{v}_n}{\| \underline{v}_n\|_D^2},
	\end{equation}
	where we use the norm $\| x\|_D:=\big(\sum_{i=1}^N |D_i| x_i^2\big)^{1/2}$ for $x\in\mathbb{R}^N$. The fact that $\tau_n\geq0$ follows from $C$ being symmetric, since then $\underline{v}_n\cdot C J C\underline{v}_n=(C\underline{v}_n)\cdot J (C\underline{v}_n)$ and $J$ is positive semi-definite. \qed
\end{proof}

\begin{remark}
	Due to the loss of energy (\emph{e.g.} to the far field), the resonant frequencies will have negative imaginary parts. In some cases it will hold that $\tau_n=0$ for some $n$, meaning that the imaginary parts exhibit higher-order behaviour in $\delta$. For example, the imaginary part of the second (dipole) frequency for a pair of identical resonators is known to be $O(\delta^{2})$ \cite{ammari2017double}.
\end{remark}

\begin{figure}
	\centering
	\includegraphics[width=0.8\linewidth]{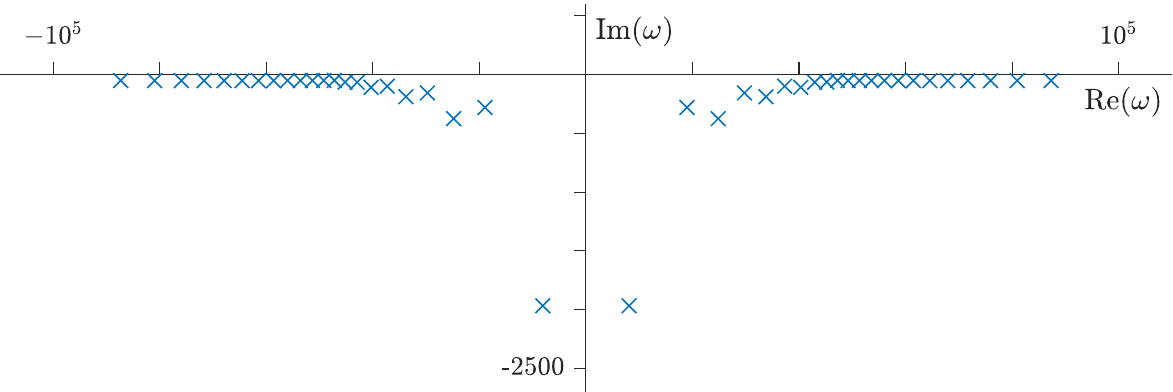}
	\caption{The subwavelength resonant frequencies $\{\omega_n^+,\omega_n^-:n=1,\dots,N\}$ lie in the lower half of the complex plane and are symmetric in the imaginary axis, shown here for an array of 22 subwavelength resonators.} \label{fig:spectrum}
\end{figure}

The resonant frequencies of an array of 22 subwavelength resonators are shown in Fig~\ref{fig:spectrum}. This array measures 35~mm, has material parameters corresponding to air-filled resonators surrounded by water and has subwavelength resonant frequencies within the range 500~Hz~--~15~kHz. Thus, this structure has similar dimensions to the human cochlea, is made from practical materials and experiences subwavelength resonance in response to frequencies that are audible to humans. Structures with these properties have been constructed \emph{e.g.} by injecting gas bubbles into polymer gels \cite{leroy2009design, leroy2009transmission}.

\begin{remark}
	The numerical simulations presented in this work were all carried out on an array of 22 cylindrical resonators. The desired quantities can be approximated by studying the two-dimensional cross section using the multipole expansion method, as outlined in \cite{davies2020hopf}.
\end{remark}

It is more illustrative to rephrase Lemma~\ref{lem:modal} in terms of basis functions that are associated with the resonant frequencies. Let $V=(v_{i,j})$ be the matrix whose columns are the eigenvectors of $\C$. Recalling Remark~\ref{rmk:Cvol_symposdef}, we have that $\C$ has a basis of eigenvectors so $V$ is invertible. Then, we define the functions
\begin{equation}
u_n(x)=\sum_{i=1}^N v_{i,n}\, \S_D[\psi_i](x),
\end{equation}
for $n=1,\dots,N$ and will seek a modal decomposition in terms of these functions. We expect the coefficients to depend on the proximity of the frequency $\omega$ to the system's resonant frequencies $\omega_n^\pm$. With this in mind, we obtain the following lemma by diagonalizing $\C$ (with the change of basis matrix $V$) and solving the resulting system. The result has been simplified further by noticing that $\omega^2-v_b^2\delta\lambda_n=(\omega-\omega_n^+)(\omega-\omega_n^-)+O(\omega^3)$ and that $e^{\i kx_1}=1+\i kx_1+\dots=1+O(\omega)$.

\begin{lemma} \label{lem:modal_res}
	If $\omega=O(\sqrt{\delta})$, the solution to the scattering problem \eqref{eq:helmholtz_equation} with incoming wave $u^{in}=Ae^{\i kx_1}$ can be written, for $x\in\mathbb{R}^3$, as
	\begin{equation*} 
	u(x)-Ae^{\i kx_1} = \sum_{n=1}^N a_n u_n(x) - \S_D\left[\S_D^{-1}[Ae^{\i kx_1}]\right](x) + O(\omega),
	\end{equation*} 
	for constants which satisfy, up to an error of order $O(\omega^3)$, the equations 
	\begin{equation*}
	a_n(\omega-\omega_n^+)(\omega-\omega_n^-)=-A\nu_n\Re(\omega_n^+)^2,
	\end{equation*}
	where $\nu_n=\sum_{j=1}^{N} [V^{-1}]_{n,j}$, {i.e.} the sum of the $n$\textsuperscript{th} row of $V^{-1}$.
\end{lemma}

From Theorem~\ref{thm:res} and Lemma~\ref{lem:modal_res}, we have results that concisely describe how the high-contrast structure behaves in the subwavelength regime. At the heart of both of these results are the eigenstates of the $N\times N$ matrix $\C$. With this in hand, we are able to make our first biomimetic deduction by adjusting the configuration of the resonator array in order to mimic the spatial frequency separation performed by the cochlea. By manipulating the structure's properties (in particular, the size gradient) the structure can reproduce the well-known relationship between incident frequency and position of maximum excitation in the cochlea. This is a crucial property for a cochlea-like device to exhibit as it is central to cochlear function. This property is a consequence of the asymmetry of the eigenmodes $u_n(x)$. The relationship between frequency and position of maximum excitation is shown in Fig~\ref{fig:tonotopic} along with the corresponding relationship from the human cochlea (see \cite{davies2019fully} for further details, also \cite{davies2020hopf,rupin2019mimicking}).

\begin{figure}
	\centering
	\includegraphics[width=0.85\linewidth]{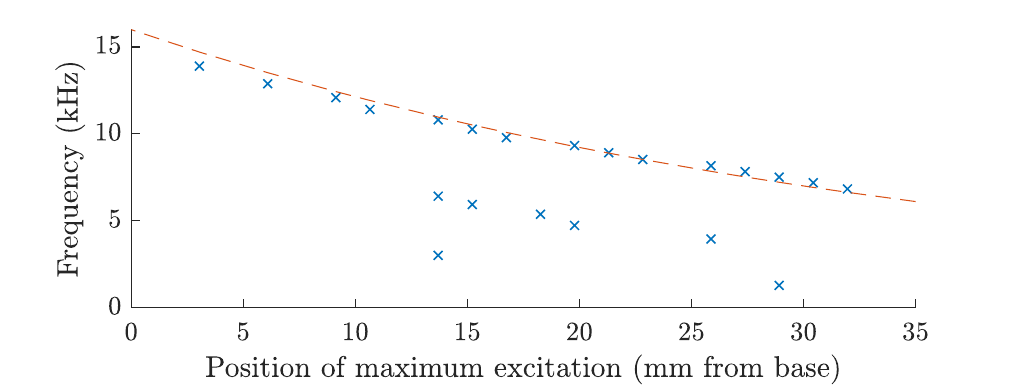}
	\caption{The relationship between frequency and position of maximum excitation in the array of subwavelength resonators can be chosen to match that which exists in the human cochlea. The resonant frequencies are shown here for an array of 22 resonators (crosses), along with the biological relationship for comparison (dashed line).} \label{fig:tonotopic}
\end{figure}

\subsection{Modal decompositions of signals} \label{sec:scattering}

Consider, now, the scattering of a more general signal, $s:[0,T]\to\mathbb{R}$, whose frequency support is wider than a single frequency and whose Fourier transform exists. Again, we assume that the wave is incident parallel to the $x_1$-axis. Consider the Fourier transform of the incoming pressure wave, given for $\omega\in\mathbb{C}$, $x\in\mathbb{R}^3$ by
\begin{align*}
u^{in}(x,\omega)&=\int_{-\infty}^{\infty} s(x_1/v-t) e^{\i\omega t}\de t\\
&=e^{\i\omega x_1/v}\hat{s}(\omega) = \hat{s}(\omega)+O(\omega),
\end{align*}
where $\hat{s}(\omega):=\int_{-\infty}^{\infty} s(u) e^{-\i\omega u}\de u$. The resulting pressure field satisfies the Helmholtz equation \eqref{eq:helmholtz_equation} along with the transmission and radiation conditions.

Working in the frequency domain, the scattered acoustic pressure field $u$ in response to the Fourier transformed signal $\hat s$ can be decomposed in the spirit of Lemma~\ref{lem:modal_res}. We write that, for $x\in\D$, the solution to the scattering problem is given by
\begin{equation} \label{eq:gen_modal_decomp} 
u(x,\omega) = \sum_{n=1}^N \frac{-\hat{s}(\omega)\nu_n\Re(\omega_n^+)^2}{(\omega-\omega_n^+)(\omega-\omega_n^-)} u_n(x) + r(x,\omega),
\end{equation} 
for some remainder $r$. We are interested in signals whose energy is mostly concentrated within the subwavelength regime. In particular, we will consider signals that are \emph{subwavelength} in the sense that
\begin{equation} \label{eq:subw_regime}
\sup_{x\in\mathbb{R}^3}\int_{-\infty}^\infty |r(x,\omega) | \de\omega = O(\delta).
\end{equation}
\begin{remark}
	The subwavelength condition \eqref{eq:subw_regime} is a strong assumption and is difficult to interpret physically. However, for the purpose of seeking to inform signal processing algorithms, which is our aim here, it is a suitable assumption.
\end{remark}

Now, we wish to apply the inverse Fourier transform to \eqref{eq:gen_modal_decomp} to obtain a time-domain decomposition of the scattered field. The condition  \eqref{eq:subw_regime} guarantees that the remainder term is not significant, while the contributions from each term in the expansion can be found through complex integration.

\begin{theorem}[Time-domain modal expansion] \label{thm:timedom}
	For $\delta>0$ and a signal $s$ which is subwavelength in the sense of the condition \eqref{eq:subw_regime}, it holds that the scattered pressure field $p(x,t)$ satisfies, for $x\in\D$, $t\in\mathbb{R}$,
	\begin{equation*} 
	p(x,t)= \sum_{n=1}^N a_n[s](t) u_n(x) + O(\delta),
	\end{equation*}
	where the coefficients are given by $a_n[s](t)=\left( s*h_n \right)(t)$ for kernels defined as
	\begin{equation} \label{eq:hdef}
	h_n(t)=
	\begin{cases}
	0, & t<0, \\
	c_n e^{\Im(\omega_n^+)t} \sin(\Re(\omega_n^+)t), & t\geq0,
	\end{cases}
	\end{equation}
	for $c_n=\nu_n\Re(\omega_n^+)$.
\end{theorem}
\begin{proof}
Applying the inverse Fourier transform to the modal expansion \eqref{eq:gen_modal_decomp} under the assumption \eqref{eq:subw_regime} yields 
	\begin{equation*} 
	p(x,t)= \sum_{n=1}^N a_n[s](t) u_n(x) + O(\delta),
	\end{equation*} 
	where, for $n=1,\dots,N$, the coefficients are given by
	\begin{equation*}
	a_n[s](t)=\frac{1}{2\pi}\int_{-\infty}^{\infty}\frac{-\hat{s}(\omega)\nu_n\Re(\omega_n^+)^2}{(\omega-\omega_n^+)(\omega-\omega_n^-)} e^{-i\omega t} \de \omega = \left( s*h_n \right)(t),
	\end{equation*}
	where $*$ denotes convolution and the kernels $h_n$ are defined for $n=1,\dots,N$ by
	\begin{equation} \label{eq:h_int_defn}
	h_n(t)=\frac{1}{2\pi}\int_{-\infty}^\infty \frac{-\nu_n\Re(\omega_n^+)^2}{(\omega-\omega_n^+)(\omega-\omega_n^-)} e^{-\i\omega t} \de \omega.
	\end{equation} 
	
	We can use complex integration to evaluate the integral in \eqref{eq:h_int_defn}. For $R>0$, let $\Con_R^\pm$ be the semicircular arc of radius $R$ in the upper $(+)$ and lower $(-)$ half-plane and let $\Con^\pm$ be the closed contour $\Con^\pm=\Con_R^\pm\cup[-R,R]$. Then, we have that
	\begin{equation*}
	h_n(t) = \frac{1}{2\pi}\oint_{\Con^\pm} \frac{-\nu_n\Re(\omega_n^+)^2}{(\omega-\omega_n^+)(\omega-\omega_n^-)} e^{-\i\omega t} \de\omega - \frac{1}{2\pi}\int_{\Con_R^\pm} \frac{-\nu_n\Re(\omega_n^+)^2}{(\omega-\omega_n^+)(\omega-\omega_n^-)} e^{-\i\omega t} \de\omega.
	\end{equation*}
	The integral around $\Con^\pm$ is easy to evaluate using the residue theorem, since it has simple poles at $\omega_n^\pm$. We will make the choice of $+$ or $-$ so that the integral along $\Con_R^\pm$ converges to zero as $R\to\infty$. For large $R$ we have a bound of the form
	\begin{equation} \label{eq:bound}
	\left|\int_{\Con_R^\pm} \frac{-\nu_n\Re(\omega_n^+)^2}{(\omega-\omega_n^+)(\omega-\omega_n^-)} e^{-\i\omega t} \de\omega\right|\leq C_n R^{-1} \sup_{\omega\in\Con_R^\pm} e^{\Im(\omega)t},
	\end{equation}
	for a positive constant $C_n$.
	
	Suppose first that $t<0$. Then we choose to integrate over $\Con_R^+$ in the upper complex plane so that \eqref{eq:bound} converges to zero as $R\to\infty$. Thus, we have for $t<0$ that 
	\begin{equation*}
	h_n(t) = \frac{1}{2\pi}\oint_{\Con^+} \frac{-\nu_n\Re(\omega_n^+)^2}{(\omega-\omega_n^+)(\omega-\omega_n^-)} e^{-\i\omega t} \de\omega =0,
	\end{equation*}
	since the integrand is holomorphic in the upper half plane. Conversely, if $t\geq0$ then we should choose to integrate over $\Con_R^-$ in order for \eqref{eq:bound} to disappear. Then, we see that for $t\geq0$ it holds that
	\begin{align*}
	h_n(t) &= \frac{1}{2\pi}\oint_{\Con^-} \frac{-\nu_n\Re(\omega_n^+)^2}{(\omega-\omega_n^+)(\omega-\omega_n^-)} e^{-\i\omega t} \de\omega \\
	&=\i\,\mathrm{Res}\left(\frac{-\nu_n\Re(\omega_n^+)^2}{(\omega-\omega_n^+)(\omega-\omega_n^-)} e^{-\i\omega t},\omega_n^+\right) + \i\,\mathrm{Res}\left(\frac{-\nu_n\Re(\omega_n^+)^2}{(\omega-\omega_n^+)(\omega-\omega_n^-)} e^{-\i\omega t},\omega_n^-\right).
	\end{align*}
	Simplifying the expressions for the residues at the two simple poles gives the result. \qed
\end{proof}

\begin{remark}
	The fact that $h_n(t)=0$ for $t<0$ ensures the causality of the modal expansion in Theorem~\ref{thm:timedom}. 
\end{remark}

Understanding the behaviour of the system of coupled subwavelength resonators as a function of time allows us to examine other properties of the cochlea-like array. For example, the asymmetry of the spatial eigenmodes $u_n(x)$ means that the decomposition from Theorem~\ref{thm:timedom} replicates the cochlea's famous travelling wave behaviour. That is, in response to an impulse the position of maximum amplitude moves slowly from left to right in the array, see \cite{davies2019fully} for details. Difficulties in reproducing this behaviour have, historically, been cited as a shortcoming of cochlear models based on local resonance \cite{bell2012resonance}.

\section{Biomimetic signal transform} \label{sec:transforms}

\begin{figure}
	\centering
	\includegraphics[width=0.9\linewidth]{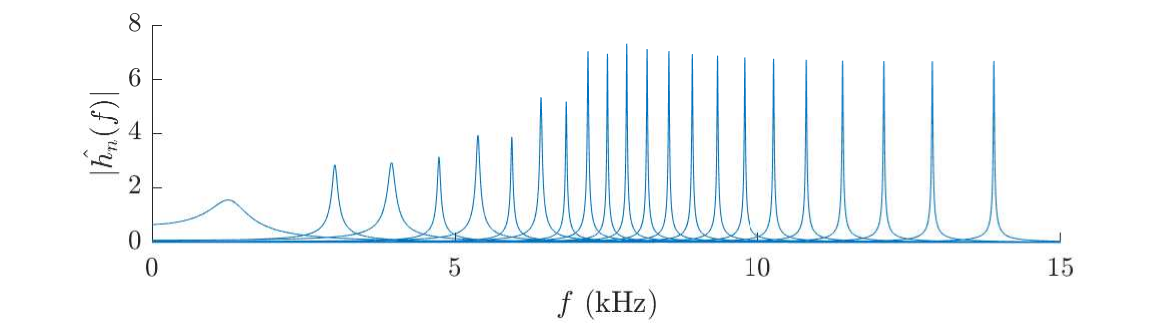}
	\caption{The band-pass filters deduced from the bio-inspired metamaterial have frequency supports that cover a range of frequencies that are audible to humans, shown here for the case of 22 resonators. The properties of these filters (centre frequency and bandwidth) are determined by the corresponding resonant frequencies $\omega_n^\pm\in\mathbb{C}$.} \label{fig:bandpass}
\end{figure}

From the asymptotic analysis in Section~\ref{sec:metamaterial}, we know how a subwavelength sound behaves when it is scattered by a high-contrast metamaterial and we used this theory to design a device that mimics the action of the cochlea. In this section, we explore using this as the basis for a biomimetic signal processing algorithm. Modelling the cochlea directly is challenging, so it is prudent to use our cochlea-like array of resonators. This mimics the cochlea's key properties and, due to its artificial nature, is much more amenable to precise mathematical analysis, as we saw in the previous section. 

From Theorem~\ref{thm:timedom}, we know that the pressure field scattered by the cochlea-like array of resonators is described by a modal decomposition whose coefficients take the form of convolutions with the functions $h_n$. We wish to explore the properties of this decomposition, given for a sound $t\mapsto s(t)$ by
\begin{equation} \label{def:subw_scat_trans}
a_n[s](t) = \left( s*h_{n}\right) (t), \quad n=1,\dots,N.
\end{equation}
Convolutional signal processing algorithms of this type have been explored in detail \cite{lyon2017human, mallat2016understanding}. Here, we will present just a few basic properties, to give some insight into the features of the algorithm that is deduced from our biomimetic approach.

\subsection{Biomimetic properties}

Since the resonant frequencies all have negative imaginary parts, each $h_n$ is a windowed oscillatory mode that acts as a band pass filter centred at $\Re(\omega_n^+)$. The frequency support of the 22 filters derived from our array of 22 resonators is shown in Fig.~\ref{fig:bandpass}. Since the imaginary part of the lowest frequency is much larger than the others (see Fig.~\ref{fig:spectrum}), $h_1$ acts somewhat as a low-pass filter.

\begin{figure}
	\centering
	\includegraphics[width=0.6\linewidth]{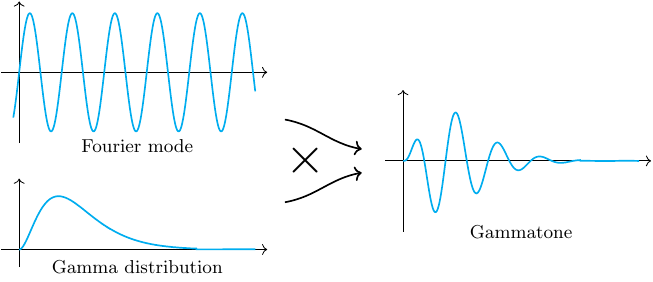}
	\caption{The cochlea-like metamaterial developed here has a response given by convolution with gammatones. A gammatone is an oscillating Fourier mode multiplied by a gamma distribution. This suggests using these gammatone functions as the basis for biomimetic signal processing algorithms.}
	\label{fig:gammatones_sketch}
\end{figure}

The basis functions $h_n$ take specific forms known as \emph{gammatones}. A gammatone is a sinusoidal mode windowed by a gamma distribution:
\begin{equation} \label{gammatone}
g(t;m,\omega,\phi)= t^{m-1} e^{\Im(\omega)t}\cos(\Re(\omega) t-\phi), \qquad t\geq0,
\end{equation}
for some order $m\in\mathbb{N}^+$ and constants $\omega\in\{z\in\mathbb{C}:\Im(z)<0\}$, $\phi\in\mathbb{R}$. We notice that $h_n(t)=c_n g(t;1,\omega_n^+,\pi/2)$. A gammatone is sketched in Fig.~\ref{fig:gammatones_sketch}.

Gammatones have been widely used to model auditory filters. Filters with gammatone kernels have been shown to approximate auditory function well, matching relatively well with physiological data and cochlear modelling \cite{patterson1988efficient, hewitt1994computer, bell2018cochlear}. They are also relatively straightforward to analyse and implement \cite{lyon2017human}, as we shall see below. The decomposition given by the metamaterial studied here, as stated in Theorem~\ref{thm:timedom}, includes only first-order gammatones. However, higher-order gammatones appear when the filters are cascaded with one another, as is a common approach for designing auditory processing algorithms \cite{lyon2017human} and convolutional networks in general \cite{mallat2016understanding}, see Appendix~\ref{app:cascade} for details.

\subsection{Stability and continuity properties} \label{sec:stab_cont}

The functions $h_n$ are bounded and continuous, meaning that if $s\in L^1(\mathbb{R})$ then $s*h_n\in L^\infty(\mathbb{R})$. If, moreover, $s$ is compactly supported then the decay properties of $h_n$ mean that $s*h_n\in L^p(\mathbb{R})$ for any $p\in[1,\infty]$. Further, we have the following lemmas which characterise the continuity and stability of $s\mapsto s*h_n$.

\begin{lemma}[Continuity of representation] \label{lem:cty}
	Consider the biomimetic scattering transform coefficients given by \eqref{def:subw_scat_trans}. There exists a positive constant $C_1$ such that for any $n\in\{1,\dots,N\}$ and any signals $s_1,s_2 \in L^1(\mathbb{R})$ it holds that
	\begin{equation*}
	\|a_n[s_1] - a_n[s_2]\|_{\infty}\leq C_1 \|s_1-s_2\|_1.
	\end{equation*}
\end{lemma}
\begin{proof}
	It holds that
	\begin{equation*}
	C_1:= \sup_{n\in\{1,\dots,N\}} \sup_{x\in\mathbb{R}}(1-c)\left|h_n(x)\right|<\infty.
	\end{equation*}
	Then, the result follows from the fact that
	\begin{equation*}
	\left|a_n[s_1](t) - a_n[s_2](t)\right|\leq \int_{-\infty}^\infty |s_1(u)-s_2(u)|\left|h_n(t-u)\right|\de u,
	\end{equation*}
	for any $t\in\mathbb{R}$. \qed
\end{proof}

The continuity property stated in Lemma~\ref{lem:cty} implies, in particular, that the representation of a signal is stable with respect to additive noise. An additional useful property is for a representation to be stable with respect to time dilations, \emph{i.e.} with respect to composition with the operator $T_\tau f (t) = f(t+\tau(t))$, where $\tau$ is some appropriate function of time.

\begin{lemma}[Pointwise stability to time warping] \label{lem:pt_stab}
	Consider the biomimetic scattering transform coefficients given by \eqref{def:subw_scat_trans}. For $\tau\in C^0(\mathbb{R};\mathbb{R})$, let $T_\tau$ be the associated time warping operator, given by $T_\tau f (t) = f(t+\tau(t))$.  Then, there exists a positive constant $C_2$ such that for any $n\in\{1,\dots,N\}$ and any signal $s \in L^1(\mathbb{R})$ it holds that
	\begin{equation*}
	\left\| a_n[s] - a_n[T_\tau s]\right\|_\infty \leq C_2 \|s\|_1 \|\tau\|_\infty.
	\end{equation*}
\end{lemma}
\begin{proof}
	Let $h_n'$ denote the first derivative of $h_n$ on $(0,\infty)$. Then, we see that
	\begin{equation*}
	C_2:=\sup_{n\in\{1,\dots,N\}} \sup_{x\in(0,\infty)}\left|h_n'(x)\right|<\infty,
	\end{equation*}
	and, by the mean value theorem, that for $t\in\mathbb{R}$
	\begin{equation*}
	\left|h_n (t-\tau(t)) - h_n(t)\right|\leq C_2 |\tau(t)|.
	\end{equation*}
	Thus, we see that for any $t\in\mathbb{R}$
	\begin{align*}
	|a_n[s] - a_n[T_\tau s]| &\leq \int_{-\infty}^\infty |s(t-u)|\left|h_n(u)-h_n(u-\tau(u))\right|\de u,\\
	&\leq C_2 \|\tau(u)\|_\infty \int_{-\infty}^\infty |s(t-u)|\de u.
	\end{align*} \qed
\end{proof}

We can improve on the notion of stability from Lemma~\ref{lem:pt_stab} by taking temporal averages of the coefficients. A particular advantage of such an approach is that it gives outputs that are invariant to translation (\emph{cf.} the motivation behind the design of the scattering transform \cite{bruna2013invariant, mallat2012group}). Let $\langle a_n[s] \rangle_{(t_1,t_2)}$ denote the average of $a_n[s](t)$ over the interval $(t_1,t_2)$, given by
\begin{equation} \label{eq:temporal_average}
\langle a_n[s] \rangle_{(t_1,t_2)}=\frac{1}{t_2-t_1}\int_{t_1}^{t_2} a_n[s](t) \de t.
\end{equation}
Lemma~\ref{lem:stab} shows that temporal averages are approximately invariant to translations if the length of the window is large relative to the size of the translation (\emph{i.e.} if $t_2-t_1\gg \|\tau\|_\infty$).

\begin{lemma}[Stability of averages to time warping] \label{lem:stab}
	Consider the biomimetic scattering transform coefficients given by \eqref{def:subw_scat_trans}. For $\tau\in C^1(\mathbb{R};\mathbb{R})$, let $T_\tau$ be the associated time warping operator, given by $T_\tau f (t) = f(t+\tau(t))$. Suppose that $\tau$ is such that $\|\tau'\|_\infty<\frac{1}{2}$. Then, there exists a positive constant $C_3$ such that for any $n\in\{1,\dots,N\}$, any time interval $(t_1,t_2)\subset\mathbb{R}$ and any signal $s \in L^1(\mathbb{R})$ it holds that
	\begin{equation*}
	\left|\langle a_n[s] \rangle_{(t_1,t_2)} - \langle a_n[T_\tau s] \rangle_{(t_1,t_2)} \right| \leq C_3 \|s\|_1 \left( \frac{2}{t_2-t_1}\|\tau\|_\infty + \|\tau'\|_\infty \right).
	\end{equation*}
\end{lemma}
\begin{proof}
	Since $\|\tau'\|_\infty\leq c<1$, $\varphi(t)=t-\tau(t)$ is invertible and $\|\varphi'\|_\infty\geq1-c$, it holds that
	\begin{align*}
	\int_{t_1}^{t_2}\left(h_n(t-\tau(t))-h_n(t)\right) \de t
	&=\int_{\varphi(t_1)}^{\varphi(t_2)}h_n(t)\frac{1}{\varphi'(\varphi^{-1}(t))}\de t-\int_{t_1}^{t_2}h_n(t) \de t \\
	& = \int_{I_1-I_2}h_n(t)\frac{1}{\varphi'(\varphi^{-1}(t))}\de t + \int_{t_1}^{t_2}h_n(t)\frac{\tau'(\varphi^{-1}(t))}{\varphi'(\varphi^{-1}(t))} \de t,
	\end{align*}
	for some intervals $I_1,I_2\subset\mathbb{R}$, each of which has length bounded by $\|\tau\|_\infty$. Now, define the constant
	\begin{equation*}
	C_3:=\sup_{n\in\{1,\dots,N\}} \sup_{x\in(0,\infty)}(1-c)^{-1}\left|h_n(x)\right|<\infty.
	\end{equation*}
	Finally, we can compute that
	\begin{align*}
	&\langle a_n[s] \rangle_{(t_1,t_2)} - \langle a_n[T_\tau s] \rangle_{(t_1,t_2)} \\
	&\hspace{4em}= \frac{1}{t_2-t_1} \int_{-\infty}^\infty s(u) \int_{t_1}^{t_2}\left(h_n(t-u-\tau(t))-h_n(t-u)\right) \de t \de u \\
	&\hspace{4em}= \frac{1}{t_2-t_1} \int_{-\infty}^\infty s(u) \bigg( \int_{I_1-I_2}h_n(t-u)\frac{1}{\varphi'(\varphi^{-1}(t-u))}\de t \\ 
	& \hspace{16em}+ \int_{t_1}^{t_2}h_n(t-u)\frac{\tau'(\varphi^{-1}(t-u))}{\varphi'(\varphi^{-1}(t-u))} \de t \bigg) \de u,
	\end{align*}
	meaning that
	\begin{align*}
	\left|\langle a_n[s] \rangle_{(t_1,t_2)} - \langle a_n[T_\tau s] \rangle_{(t_1,t_2)} \right| &\leq \frac{1}{t_2-t_1} \|s\|_1 \Big[2\|\tau\|_\infty C_3+(t_2-t_1)C_3\|\tau'\|_\infty\Big]. 
	\end{align*} \qed
\end{proof}

Taken together, Lemmas~\ref{lem:cty}~to~\ref{lem:stab} show that the biomimetic scattering transform proposed here, in \eqref{def:subw_scat_trans}, which was derived from the behaviour of the cochlea-inspired metamaterial studied in Section~\ref{sec:metamaterial}, has favourable stability and continuity properties. This is in addition to the fact that gammatone kernels match direct observations of the human auditory system well, suggesting that this structure provides an excellent starting point for a biomimetic signal processing architecture.

\section{Representation of natural sounds} \label{sec:natsounds}

In Section~\ref{sec:metamaterial} and Section~\ref{sec:transforms} we used asymptotic methods to demonstrate how properties of signal filtering acoustic metamaterials can be deduced from biological auditory systems. These links are labelled (a) and (b) in Fig.~\ref{fig:workflow}. In this final section, we will present an example of how our asymptotic framework can be used to facilitate link (c), whereby additional properties of signal processing algorithms are deduced from the target biological systems. In particular, we will devise an approach to replicate the human auditory system's ability to recognise long-range statistical properties of sounds. This is intended as a creative example of the innovation that is possible, given the framework that has been developed here.

The human auditory system does not just extract information from a signal locally in time, it is also able to recognise global properties of a sound and can appreciate notions of timbre and quality. We would like to design an approach that can account for this, by adding an additional processing step to the biomimetic signal transform that was derived above. The approach set out in this section is tailored to the class of \emph{natural sounds}. This is, a class of natural and behaviourally significant sounds to which humans are known to be adapted. These sounds have observed statistical properties that we are able to exploit (and which we interpret as their defining characteristic). The class of natural sounds is very broad and includes animal sounds, speed and most music (this is by design, since music satisfying these properties is generally more pleasing to listen to \cite{voss1978music}).

\subsection{Properties of natural sounds} \label{sec:nat_properties}

Let us briefly summarise what has been observed about the low-order statistics of natural sounds \cite{attias1997temporal, attias1998coding, theunissen2014neural, voss19751}. For a sound $s(t)$, let $a_\omega(t)$ be the component at frequency $\omega$ (obtained \emph{e.g.} through the application of a band-pass filter centred at $\omega$). Then we can write that
\begin{equation}
a_\omega(t) = A_\omega(t) \cos(\omega t + \phi_\omega(t)),
\end{equation}
where $A_\omega(t)\geq0$ and $\phi_\omega(t)$ are the instantaneous amplitude and phase, respectively. We view $A_\omega(t)$ and $\phi_\omega(t)$ as stochastic processes and wish to understand their statistics. 

The most famous characteristic of natural sounds is that several properties of their frequency components vary according to the inverse of the frequency. In particular, it is well known that the power spectrum (the square of the Fourier transform) of the amplitude satisfies a relationship of the form
\begin{equation} \label{eq:SA}
S_{A_\omega}(f)=|\hat{A}_\omega(f)|^2\propto \frac{1}{f^\gamma}, \qquad 0<f<f_{max},
\end{equation}
for a positive parameter $\gamma$ (which often lies in a neighbourhood of 1) and some maximum frequency $f_{max}$. Further, this property is independent of $\omega$, \emph{i.e.} of the frequency band that is studied \cite{attias1997temporal}.

Consider the log-amplitude, $\log_{10} A_\omega(t)$. It has been observed that for a variety of natural sounds (including speech, animal vocalisations, music and environmental sounds) the log-amplitude is locally stationary, in the sense that it satisfies a statistical distribution that does not depend on the time $t$. Suppose we normalise the log-amplitude so that it has zero mean and unit variance, giving a quantity that is invariant to amplitude scaling. Then, the normalised log-amplitude averaged over some time interval $[t_1,t_2]$ satisfies a distribution of the form \cite{attias1998coding}
\begin{equation} \label{eq:pA}
p_A(x) = \beta \exp(\beta x - \alpha -e^{\beta x-\alpha}),
\end{equation}
where $\alpha$ and $\beta$ are real-valued parameters and $\beta>0$.
Further, this property does not depend on the frequency band and is scale invariant in the sense that it is independent of the time interval over which the average is taken.

The properties of the instantaneous phase $\phi_\omega(t)$ have also been studied. Similar to the instantaneous amplitude, the power spectrum $S_{\phi_\omega}$ of the instantaneous phase satisfies a $1/f$-type relationship, of the same form as \eqref{eq:SA}. On the other hand, the instantaneous phase is non-stationary (even locally), making it difficult to describe through the above methods. A more tractable quantity is the instantaneous frequency, defined as
\begin{equation*}
\lambda_\omega = \dd{\phi_\omega}{t}.
\end{equation*}
It has been observed that $\lambda_\omega(t)$ is locally stationary for natural sounds and the temporal mean of its modulus satisfies a distribution $p_\lambda$ of the form \cite{attias1997temporal}
\begin{equation} \label{eq:pphi}
p_\lambda(x) \propto (\zeta^2+x^2)^{-\eta/2},
\end{equation}
for positive parameters $\zeta$ and $\eta>1$. Examples of these distributions are shown in Fig.~\ref{fig:trumpet} for a short recording of a trumpet playing a single note.

\begin{figure}
	\begin{center}
		\begin{subfigure}[b]{0.45\linewidth}
			\includegraphics[width=\linewidth]{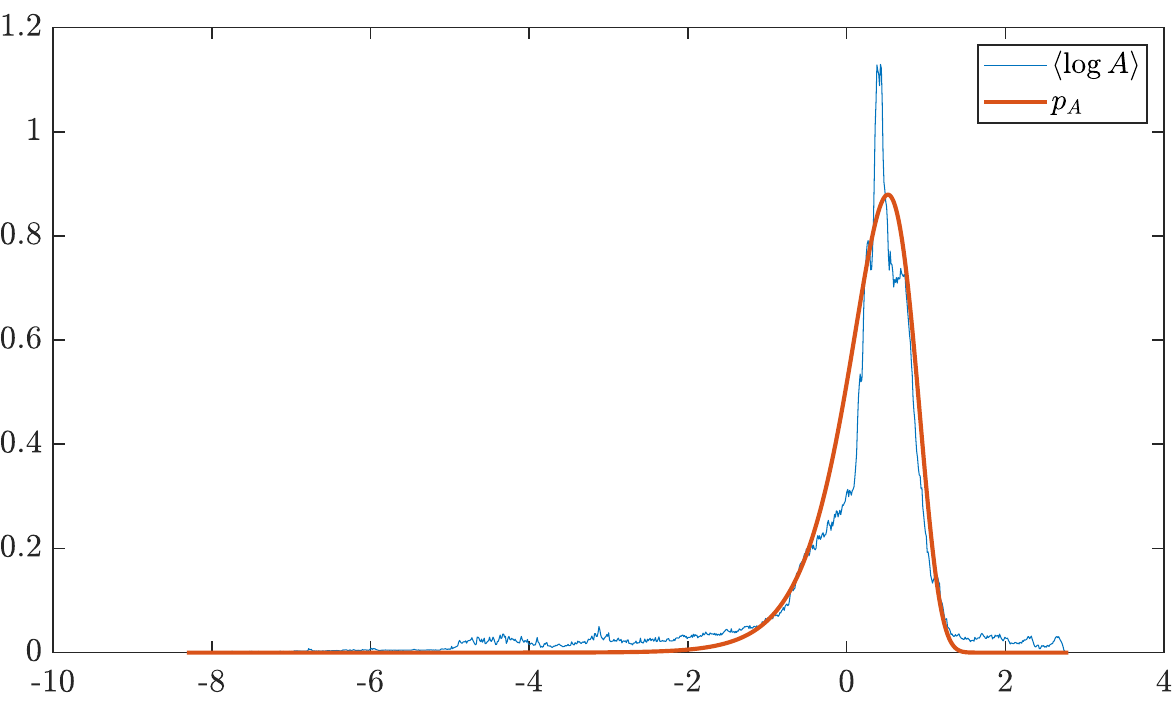}
			\caption{A histogram of the normalised time-averaged log-amplitude, fitted to a distribution of the form \eqref{eq:pA}.}
		\end{subfigure}
		\hspace{0.6cm}
		\begin{subfigure}[b]{0.45\linewidth}
			\includegraphics[width=\linewidth]{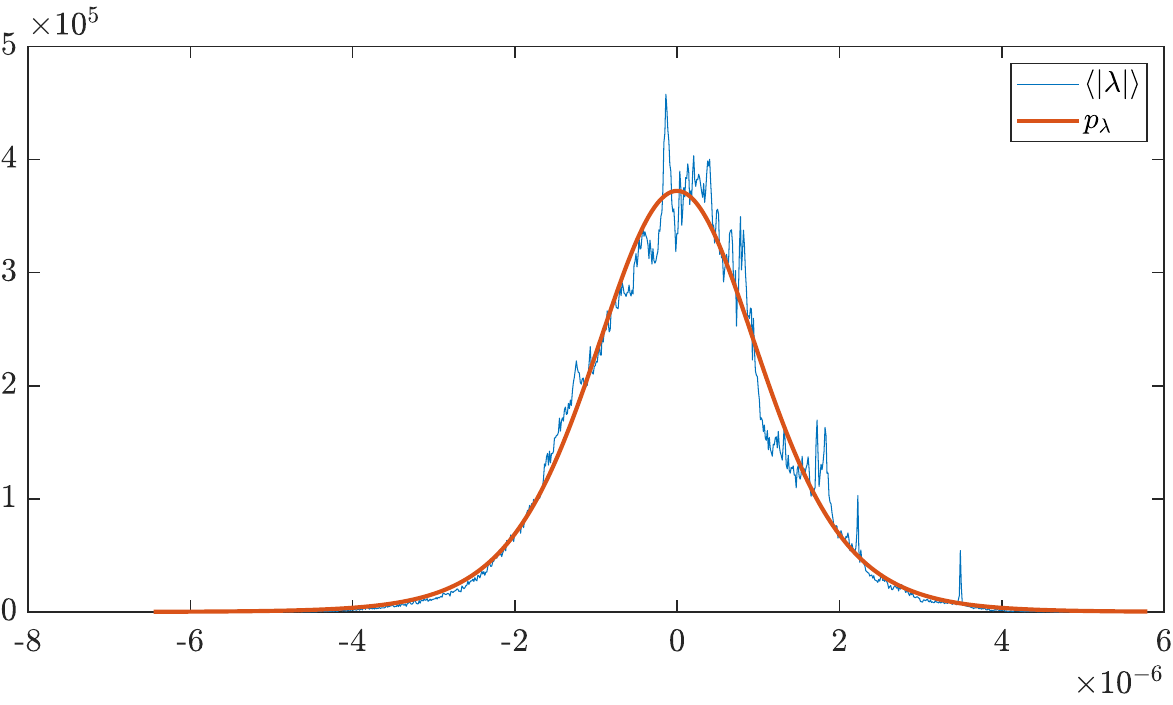}
			\caption{A histogram of the normalised time-averaged instantaneous frequency, fitted to a distribution of the form \eqref{eq:pphi}.}
		\end{subfigure}
		\begin{subfigure}[b]{0.45\linewidth}
			\includegraphics[width=\linewidth]{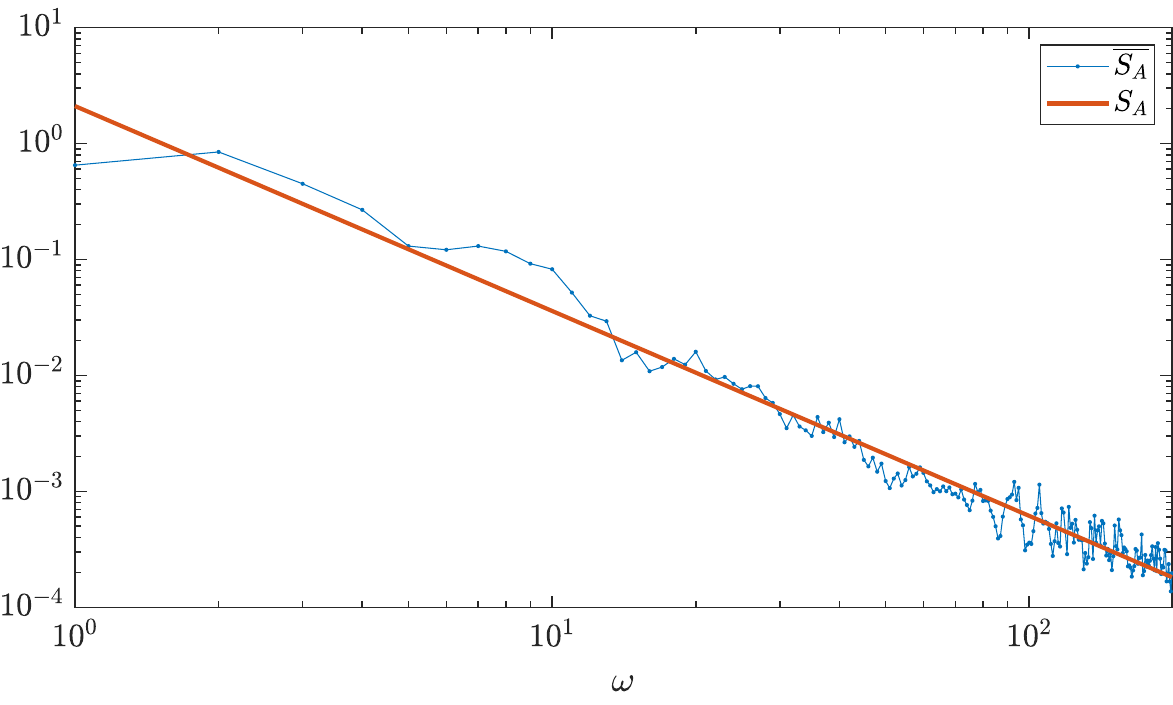}
			\caption{The averaged (over different frequency bands) power spectra of the instantaneous amplitude, fitted to a distribution of the form \eqref{eq:SA}.}
		\end{subfigure}
		\hspace{0.6cm}
		\begin{subfigure}[b]{0.45\linewidth}
			\includegraphics[width=\linewidth]{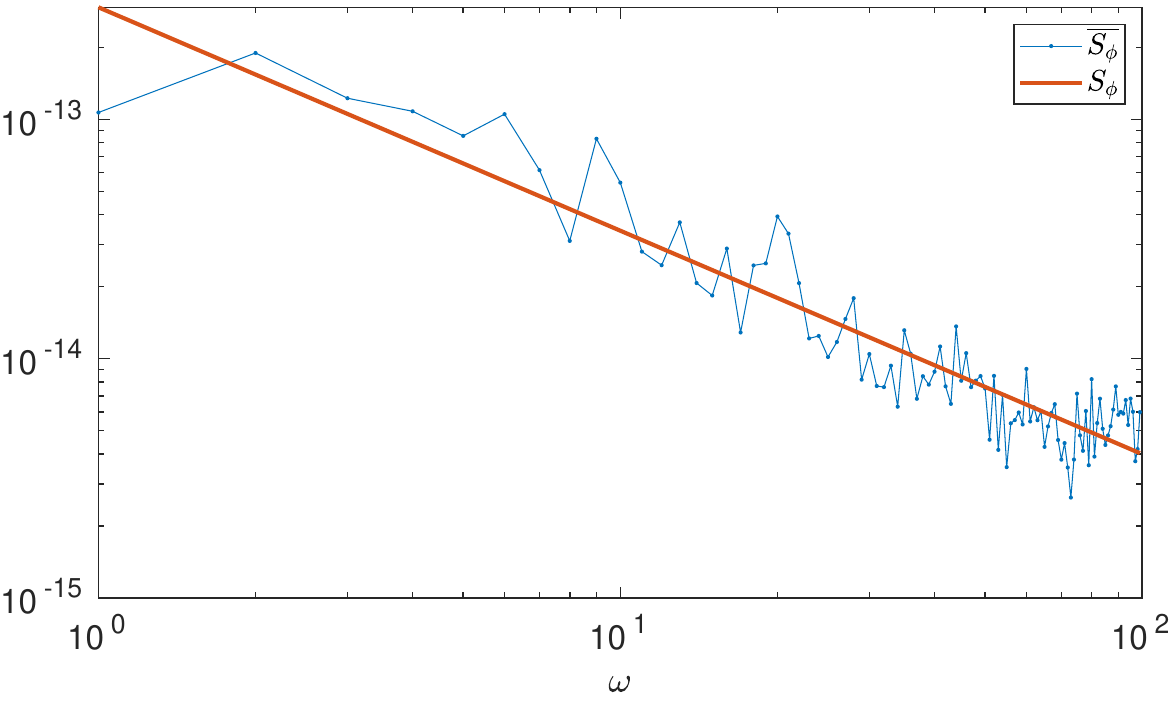}
			\caption{The averaged (over different frequency bands) power spectra of the instantaneous phase, fitted to a distribution of the form \eqref{eq:SA}.}
		\end{subfigure}
		\caption{Natural sounds have been observed to satisfy given low-order statistical properties. The output of scattering by a cochlea-like metamaterial will satisfy known distributions, as a result. Here, the data extracted by the cochlea-like device from a recording of a trumpet playing a single note are shown along with the fitted distributions.} \label{fig:trumpet}
	\end{center}
\end{figure}

\subsection{Representation algorithm} \label{sec:algorithm}

For a given natural sound, we wish to find the parameters that characterise its global properties, according to \eqref{eq:SA}--\eqref{eq:pphi}. Given a signal $s$ we first compute the convolution with the band-pass filter $h_n$ to yield the spectral component at the frequency $\Re(\omega_n^+)$, given by
\begin{equation*}
a_n[s](t)=A_n(t)\cos(\Re(\omega_n^+)t+\phi_n(t)).
\end{equation*}
We extract the functions $A_n$ and $\phi_n$ from $a_n[s]$ using the Hilbert transform \cite{attias1997temporal,flanagan1980parametric,boashash1992estimating}. In particular, we have that
\begin{equation*}
a_n[s](t)+\i H(a_n[s])(t)=a_n[s](t)+\frac{\i}{\pi}\int_{-\infty}^\infty \frac{a_n[s](u)}{t-u}\de u = A_n(t) e^{\i(\Re(\omega_n^+)t+\phi_n(t))},
\end{equation*}
from which we can extract $A_n$ and $\phi_n$ by taking the complex modulus and argument, respectively. 

\begin{remark}
	It is not obvious that the Hilbert transform $H(a_n[s])$ is well-defined. Indeed, we must formally take the principal value of the integral. For a signal that is integrable and has compact support, $H(a_n[s])(t)$ exists for almost all $t\in\mathbb{R}$.
\end{remark}

\begin{figure}
	\centering
	\includegraphics[width=0.6\linewidth]{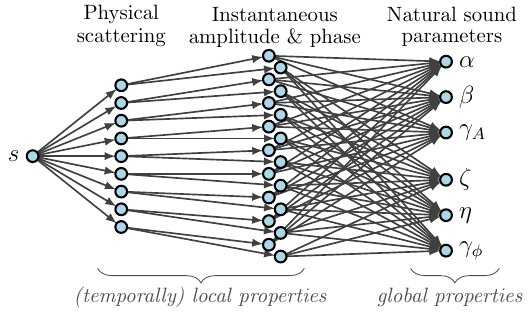}
	\caption{The architecture considered in this work first applies filters derived from physical scattering by a cochlea-like device, then extracts the instantaneous amplitude and phase before, finally, estimating the parameters of the associated natural sound distributions.}
	\label{fig:architecture}
\end{figure}

Given the functions $A_n$ and $\phi_n$, the power spectra $S_{A_n}(f)$ and $S_{\phi_n}(f)$ can be computed by applying the Fourier transform and squaring. We estimate the relationships of the form \eqref{eq:SA} by first averaging the $N$ power spectra, to give $\overline{S}_A(f):=\frac{1}{N}\sum_n {S}_{A_n}(f)$ and $\overline{S}_\phi(f):=\frac{1}{N}\sum_n {S}_{\phi_n}(f)$ before fitting curves $f^{-\gamma_A}$ and $f^{-\gamma_\phi}$ using least-squares regression. We estimate the parameters of the probability distributions \eqref{eq:pA} and \eqref{eq:pphi} by normalising both $\log_{10}A_n(t)$ and $\lambda_n(t)$ so that
\begin{equation*}
\langle\log_{10}A_n\rangle=0, \qquad  \langle(\log_{10}A_n)^2\rangle=1,
\end{equation*}
and similarly for $\lambda_n(t)$, before repeatedly averaging the normalised functions over intervals $[t_1,t_2]\subset\mathbb{R}$. Curves of the form \eqref{eq:pA} and \eqref{eq:pphi} are then fitted to the resulting histograms (which combine the temporal averages from different filters $n=1,\dots,N$ and different time intervals $[t_1,t_2]$) using non-linear least-squares optimisation. A schematic of this parameter extraction architecture is given in Fig.~\ref{fig:architecture}. An example of the four datasets and their fitted distributions are shown in Fig.~\ref{fig:trumpet} for a recording of a trumpet. Table~\ref{table} shows some other examples of these parameters for various natural sounds.

\begin{table}
	\centering
	\begin{tabular}{ c c c c c c c c c}
		& \rotatebox{90}{trumpet} & \rotatebox{90}{violin} & \rotatebox{90}{cello} & \rotatebox{90}{thunder} & \rotatebox{90}{baby speech} & \rotatebox{90}{adult speech} & \rotatebox{90}{running water} & \rotatebox{90}{crow call} \\ \hline
		$\gamma_A$ & 1.767 & 1.563 & 1.528 & 1.415 & 1.763 & 1.808 & 1.466 & 1.571\\  
		$\alpha$ & 1.244 & 0.375 & 0.284 & 0.474 & 0.517 & 0.528 & 0.336 & 0.649\\  
		$\beta$ & 2.390 & 0.783 & 0.841 & 0.596 & 0.747 & 0.894 & 0.484 & 0.896\\
		$\gamma_\phi$ & 0.763 & 0.871 & 0.6977 & 0.446 & 1.192 & 1.125 & 1.088 & 0.908\\ 
		$\zeta$ {$(\times10^{-6})$} & 2.878 & 3.433 & 6.1149 & 6.322 & 4.773 & 5.176 & 5.200 & 4.212  \\  
		$\eta$ & 8.579 & 11.824 & 8.679 & 8.315 & 9.660 & 9.358 & 9.290 & 10.475\\ 
		\hline
	\end{tabular}
	\caption{Values of the estimated distribution parameters for different samples of natural sounds. $\gamma_A$ and $\gamma_\phi$ capture the $f^{-\gamma}$ relationships of the averaged power spectra $\overline{S}_A$ and $\overline{S}_\phi$. The distribution $p_A$ of the time-averaged, normalised log-amplitude is parametrised by $\alpha$ and $\beta$ while $\zeta$ and $\eta$ parametrise the distribution $p_\lambda$ of the time-averaged instantaneous frequency.} \label{table}
\end{table}

The widely observed properties of natural sounds give us a set of six coefficients $(\gamma_A,\alpha,\beta,\gamma_\phi,\zeta,\eta)\in\mathbb{R}^6$ that portray global properties of a sound. Further, these coefficients can be extracted from the output of the cochlea-inspired device designed in Section~\ref{sec:metamaterial}. Global parameters of this kind have been shown to capture, in a perceptually significant sense, the quality of a sound and play an important role in our ability to recognise sounds efficiently \cite{mcdermott2011sound, woolley2005tuning, lesica2008efficient, nelken1999responses}. A question for future exploration is the extent to which adding these parameters to the information already gained from the biomimetic scattering transform studied in Section~\ref{sec:transforms} improves its performance on \emph{e.g.} classification problems. What is clear, however, is that including these coefficients increases the extent to which the representation algorithm mimics the perceptual abilities of the human auditory system.

\section{Concluding remarks}

Biomimicry has already had a significant impact on both the development of artificial hearing approaches and, conversely, on our understanding of biological auditory systems. In this work, we have set out the mathematical foundations to support the development of this powerful methodology. By developing a biomimetic signal processing algorithm that replicates both the response of a cochlea-inspired metamaterial and the human auditory system's ability to recognise natural sounds, we have demonstrated the three-way exchange of ideas that is possible within this framework. The concise formulas developed in this work demonstrate the power of asymptotic analysis, which can be used to facilitate future design breakthroughs and will play a key role in answering many of the open questions about how we interact with sounds in our environment.

\appendix
\section{Cascaded biomimetic scattering filters} \label{app:cascade}
In order to reveal richer properties of an acoustic signal, a common approach is to build on the method from Section~\ref{sec:transforms} by using the filters in a convolutional neural network \cite{lyon2017human,mallat2016understanding}. That is, a cascade that alternates convolutions with $h_n$ and some activation function $\Theta$:

\begin{align}
a_{n_1}^{(1)}[s](t) &= \Theta\left( s*h_{n_1} \right)(t), \nonumber \\
a_{n_1,n_2}^{(2)}[s](t) &= \Theta\left( a_{n_1}^{(1)}[s]*h_{n_2} \right)(t), \label{eq:adef} \\[-0.6em]
& \ \, \vdots \nonumber \\[-0.5em]
a_{n_1,\dots,n_k}^{(k)}[s](t) &= \Theta\left( a_{n_1,\dots,n_{k-1}}^{(k-1)}[s]*h_{n_k} \right)(t), \nonumber
\end{align}
where, in each case, the indices are such that $(n_1,n_2,\ldots,n_k)\in\{1,\ldots,N\}^k$. The vector $P_k=(n_1,\ldots,n_k)$ is known as the \emph{path} of $a_{P_k}^{(k)}$.

\begin{figure}
	\centering
	\includegraphics[width=\linewidth]{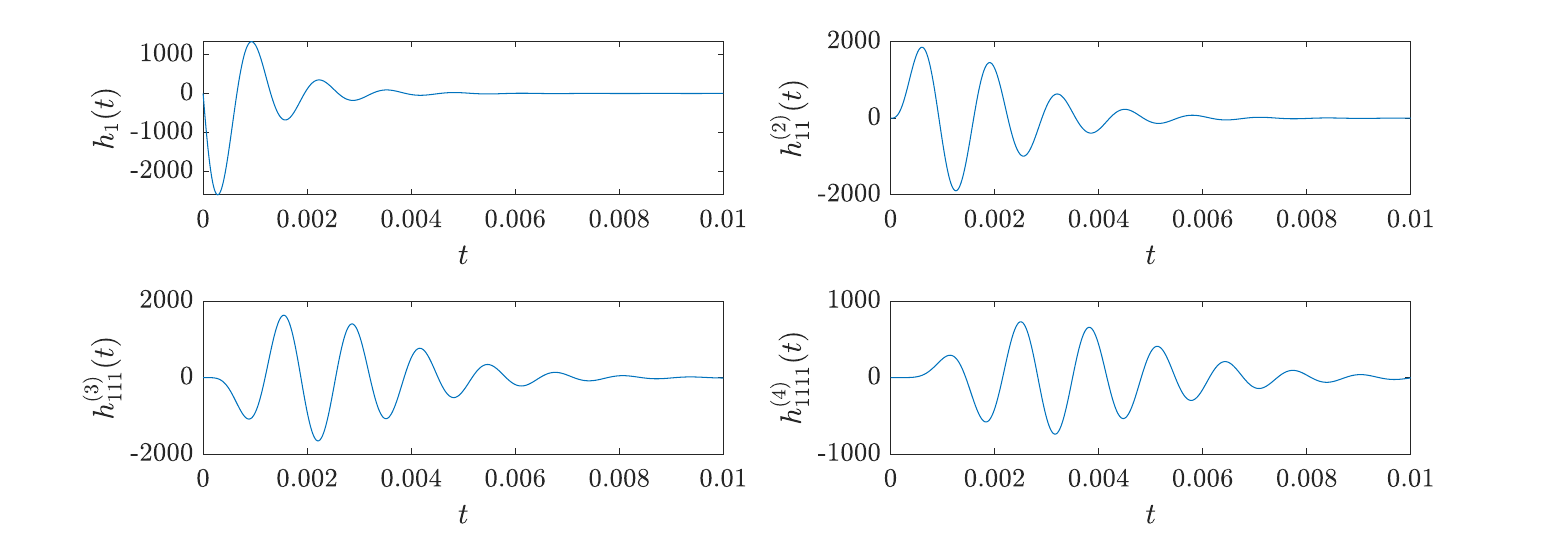}
	\caption{At successively deeper layers in a cascade of gammatone filters successively higher-order gammatones emerge. Here, increasing orders are shown for the first subwavelength resonant frequency in the case of 22 resonators.} \label{fig:gammatones}
\end{figure}

As an expository example, we consider the case where $\Theta:\mathbb{R}\to\mathbb{R}$ is the identity $Id(x)=x$. In this case, for any depth $k$ we have that $a_{P_k}^{(k)}[s]=s*h_{P_k}^{(k)}$ for some function $h_{P_k}^{(k)}$ which is the convolution of $k$ gammatones of the form \eqref{eq:hdef}, indexed by the path $P_k$. This simplification means that a more detailed mathematical analysis is possible, allowing us to briefly demonstrate some basic properties.

The basis functions $h_{P_k}^{(k)}$ take specific forms. In particular, the diagonal terms contain gammatones, as defined in \eqref{gammatone}. This is made precise by the following lemma.
  
\begin{lemma}[The emergence of higher-order gammatones]
  	For $k\in\mathbb{N}^+$ and $n\in\{1,\dots,N\}$, there exist non-negative constants $C_m^{n,k}$, $m=1,\dots,k$, such that
  	\begin{equation*}
  	h_{n,\dots,n}^{(k)}(t)=(c_n)^k\sum_{m=1}^{k}C_m^{n,k}g(t;m,\omega_n^+,m\tfrac{\pi}{2}).
  	\end{equation*}
  	In particular, $C_k^{n,k}>0$.
\end{lemma}
\begin{proof}
  	Let us write $G_n^m(t):=g(t;m,\omega_n^+,m\tfrac{\pi}{2})$, for the sake of brevity. 	Firstly, it holds that $h_n(t)=c_n G_n^1(t)$. Furthermore, we have that
  	\begin{align*}
  	(G_n^1*G_n^1)(t)&= \frac{1}{2}G_n^2(t)+\frac{1}{2\Re(\omega_n^+)}G_n^1(t),
  	\end{align*}
  	as well as, for $m\geq3$, the recursion relation
  	\begin{align*}
  	(G_n^{m-1}*G_n^1)(t)
  	&=\frac{1}{2(m-1)} G_n^{m}(t) +\frac{m-2}{2\Re(\omega_n^+)}(G_n^{m-2}*G_n^1) (t).
  	\end{align*}
  	The result follows by repeatedly applying this formula. In particular, we find that the coefficients $C_k^{n,k}$ are given by
  	\begin{equation*}
  	C_k^{n,k}=\frac{1}{2^{k-1}\,(k-1)!},
  	\end{equation*}
	which are necessarily strictly positive for any positive $k=1,2,3,\dots$.
\end{proof}

\begin{remark}
	While higher-order gammatones appeared here through the cascade of filters, gammatones also arise directly from resonator scattering if higher-order resonators are used \cite{lyon2017human}. That is, resonators that exhibit higher-order singularities in the frequency domain. It was recently shown that if sources of energy gain and loss are introduced to an array of coupled subwavelength resonators then such higher-order resonant modes can exist \cite{ammari2020exceptional, ammari2020high}.
\end{remark}

For any depth $k\in\mathbb{N}$ and path $P_k\in\{1,\dots,N\}^k$, it holds that $h_{P_k}^{(k)}$ is a bounded continuous function. This means that the stability and continuity results from Section~\ref{sec:stab_cont} can be modified to also include this setting. In each case, the appropriate result would say that given a maximum depth $k_{max}\in\mathbb{N}^+$, there exists a positive constant such that for any depth $k=1,\dots,k_{max}$ and any path $P_k\in\{1,\dots,N\}^k$ the corresponding result holds for signals in $L^1(\mathbb{R})$.

\section*{Acknowledgments}
The authors are grateful to Andrew Bell and Florian Feppon for insightful comments on early versions of this manuscript. The numerical experiments in this work were carried out on a variety of sound recordings from the Univeristy of Iowa's archive of musical instrument samples \cite{iowasamples} and the collection of natural sound stimuli that was compiled by \cite{norman2015distinct}. The code used for this study is available for download from GitHub \cite{davies2020gitrepo}.

\bibliographystyle{siamplain}
\bibliography{references}
\end{document}